\documentclass[12pt,leqno,draft]{article}
\usepackage[pagewise]{lineno}
\usepackage{amsfonts}
\usepackage{amsmath, amsthm, amsfonts, amssymb, color}
\usepackage{mathrsfs}
\usepackage{color}
\usepackage{stmaryrd}
\newtheorem{thm}{Theorem}[section]

\newtheorem{lem}[thm]{Lemma}
\newtheorem{prp}[thm]{Proposition}

\newtheorem{rem}[thm]{Remark}
\theoremstyle{definition}
\newtheorem{defn}{Definition}[section]
\newcommand{\scr}[1]{\mathscr #1}
\definecolor{wco}{rgb}{0.5,0.2,0.3}

\numberwithin{equation}{section} \theoremstyle{remark}

\newcommand{\ua}{\uparrow}

\title{{\bf   Path Dependent McKean-Vlasov SDEs with H\"{o}lder Continuous Diffusion}\footnote{ Supported in
 part by  NNSFC (11801406). } }
\author{
{\bf   Xing Huang $^{a)}$,  Xucheng Wang $^{a)}$  }\\
\footnotesize{ a)Center for Applied Mathematics, Tianjin
University, Tianjin 300072, China}\\
\footnotesize{  xinghuang@tju.edu.cn, }
\footnotesize{  wxc3018233022@163.com }}
\begin{document}
\allowdisplaybreaks
\def\R{\mathbb R}  \def\ff{\frac} \def\ss{\sqrt} \def\B{\mathbf
B}
\def\N{\mathbb N} \def\kk{\kappa} \def\m{{\bf m}}
\def\ee{\varepsilon}\def\ddd{D^*}
\def\dd{\delta} \def\DD{\Delta} \def\vv{\varepsilon} \def\rr{\rho}
\def\<{\langle} \def\>{\rangle} \def\GG{\Gamma} \def\gg{\gamma}
  \def\nn{\nabla} \def\pp{\partial} \def\E{\mathbb E}
\def\d{\text{\rm{d}}} \def\bb{\beta} \def\aa{\alpha} \def\D{\scr D}
  \def\si{\sigma} \def\ess{\text{\rm{ess}}}
\def\beg{\begin} \def\beq{\begin{equation}}  \def\F{\scr F}
\def\Ric{\text{\rm{Ric}}} \def\Hess{\text{\rm{Hess}}}
\def\e{\text{\rm{e}}} \def\ua{\underline a} \def\OO{\Omega}  \def\oo{\omega}
 \def\tt{\tilde} \def\Ric{\text{\rm{Ric}}}
\def\cut{\text{\rm{cut}}} \def\P{\mathbb P} \def\ifn{I_n(f^{\bigotimes n})}
\def\C{\scr C}   \def\G{\scr G}   \def\aaa{\mathbf{r}}     \def\r{r}
\def\gap{\text{\rm{gap}}} \def\prr{\pi_{{\bf m},\varrho}}  \def\r{\mathbf r}
\def\Z{\mathbb Z} \def\vrr{\varrho} \def\ll{\lambda}
\def\L{\scr L}\def\Tt{\tt} \def\TT{\tt}\def\II{\mathbb I}
\def\i{{\rm in}}\def\Sect{{\rm Sect}}  \def\H{\mathbb H}
\def\M{\scr M}\def\Q{\mathbb Q} \def\texto{\text{o}} \def\LL{\Lambda}
\def\Rank{{\rm Rank}} \def\B{\scr B} \def\i{{\rm i}} \def\HR{\hat{\R}^d}
\def\to{\rightarrow}\def\l{\ell}\def\iint{\int}
\def\EE{\scr E}\def\no{\nonumber}
\def\A{\scr A}\def\V{\mathbb V}\def\osc{{\rm osc}}
\def\BB{\scr B}\def\Ent{{\rm Ent}}\def\3{\triangle}\def\H{\scr H}
\def\U{\scr U}\def\8{\infty}\def\1{\lesssim}\def\HH{\mathrm{H}}
 \def\T{\scr T}
 \def\R{\mathbb R}  \def\ff{\frac} \def\ss{\sqrt} \def\B{\mathbf
B} \def\W{\mathbb W}
\def\N{\mathbb N} \def\kk{\kappa} \def\m{{\bf m}}
\def\ee{\varepsilon}\def\ddd{D^*}
\def\dd{\delta} \def\DD{\Delta} \def\vv{\varepsilon} \def\rr{\rho}
\def\<{\langle} \def\>{\rangle} \def\GG{\Gamma} \def\gg{\gamma}
  \def\nn{\nabla} \def\pp{\partial} \def\E{\mathbb E}
\def\d{\text{\rm{d}}} \def\bb{\beta} \def\aa{\alpha} \def\D{\scr D}
  \def\si{\sigma} \def\ess{\text{\rm{ess}}}
\def\beg{\begin} \def\beq{\begin{equation}}  \def\F{\scr F}
\def\Ric{\text{\rm{Ric}}} \def\Hess{\text{\rm{Hess}}}
\def\e{\text{\rm{e}}} \def\ua{\underline a} \def\OO{\Omega}  \def\oo{\omega}
 \def\tt{\tilde} \def\Ric{\text{\rm{Ric}}}
\def\cut{\text{\rm{cut}}} \def\P{\mathbb P} \def\ifn{I_n(f^{\bigotimes n})}
\def\C{\scr C}      \def\aaa{\mathbf{r}}     \def\r{r}
\def\gap{\text{\rm{gap}}} \def\prr{\pi_{{\bf m},\varrho}}  \def\r{\mathbf r}
\def\Z{\mathbb Z} \def\vrr{\varrho} \def\ll{\lambda}
\def\L{\scr L}\def\Tt{\tt} \def\TT{\tt}\def\II{\mathbb I}
\def\i{{\rm in}}\def\Sect{{\rm Sect}}  \def\H{\mathbb H}
\def\M{\scr M}\def\Q{\mathbb Q} \def\texto{\text{o}} \def\LL{\Lambda}
\def\Rank{{\rm Rank}} \def\B{\scr B} \def\i{{\rm i}} \def\HR{\hat{\R}^d}
\def\to{\rightarrow}\def\l{\ell}\def\iint{\int}
\def\EE{\scr E}\def\Cut{{\rm Cut}}
\def\A{\scr A} \def\Lip{{\rm Lip}}
\def\BB{\scr B}\def\Ent{{\rm Ent}}\def\L{\scr L}
\def\R{\mathbb R}  \def\ff{\frac} \def\ss{\sqrt} \def\B{\mathbf
B}
\def\N{\mathbb N} \def\kk{\kappa} \def\m{{\bf m}}
\def\dd{\delta} \def\DD{\Delta} \def\vv{\varepsilon} \def\rr{\rho}
\def\<{\langle} \def\>{\rangle} \def\GG{\Gamma} \def\gg{\gamma}
  \def\nn{\nabla} \def\pp{\partial} \def\E{\mathbb E}
\def\d{\text{\rm{d}}} \def\bb{\beta} \def\aa{\alpha} \def\D{\scr D}
  \def\si{\sigma} \def\ess{\text{\rm{ess}}}
\def\beg{\begin} \def\beq{\begin{equation}}  \def\F{\scr F}
\def\Ric{\text{\rm{Ric}}} \def\Hess{\text{\rm{Hess}}}
\def\e{\text{\rm{e}}} \def\ua{\underline a} \def\OO{\Omega}  \def\oo{\omega}
 \def\tt{\tilde} \def\Ric{\text{\rm{Ric}}}
\def\cut{\text{\rm{cut}}} \def\P{\mathbb P} \def\ifn{I_n(f^{\bigotimes n})}
\def\C{\scr C}      \def\aaa{\mathbf{r}}     \def\r{r}
\def\gap{\text{\rm{gap}}} \def\prr{\pi_{{\bf m},\varrho}}  \def\r{\mathbf r}
\def\Z{\mathbb Z} \def\vrr{\varrho} \def\ll{\lambda}
\def\L{\scr L}\def\Tt{\tt} \def\TT{\tt}\def\II{\mathbb I}
\def\i{{\rm in}}\def\Sect{{\rm Sect}}  \def\H{\mathbb H}
\def\M{\scr M}\def\Q{\mathbb Q} \def\texto{\text{o}} \def\LL{\Lambda}
\def\Rank{{\rm Rank}} \def\B{\scr B} \def\i{{\rm i}} \def\HR{\hat{\R}^d}
\def\to{\rightarrow}\def\l{\ell}
\def\8{\infty}\def\I{1}\def\U{\scr U}
\maketitle
\begin{abstract}
In this paper, the well-posedness for one-dimensional path dependent McKean-Vlasov SDEs with $\alpha$($\alpha\geq \frac{1}{2}$)-H\"{o}lder continuous diffusion is investigated. Moreover, the associated quantitative propagation of chaos in the sense of Wasserstein distance, total variation distance as well as relative entropy is studied.

\end{abstract} \noindent
 \vskip 2cm

\section{Introduction}
Distribution dependent SDEs can be used to characterize the nonlinear Fokker-Planck-Kolmogorov equations. They are also called McKean-Vlasov SDEs due to the pioneer work in \cite{Mc}. On the other hand, McKean-Vlasov SDE can be viewed as  the limit equation of a single particle in the mean field interacting particle system, which is related to the propagation of chaos \cite{SZ}, so it is also called mean field SDE. Recently, there are plentiful results on McKean-Vlasov SDEs. With respect to the well-posedness, one can refer to \cite{BMP,Ch,CF,HW,HW20,MV,RZ,Wangb} and references therein, see also \cite{HX} for the path dependent case with singular drifts. In \cite{Ch,CF,HW,HW20,RZ}, the diffusion is assumed to be uniformly elliptic. For the propagation of chaos, see \cite{BO,CD,DEG,FH,GLL,HRZ,JR,L,SZ,Zhang}. One can also refer to \cite{GLWZ,RW,W21} for the long time behavior of mean field interacting particle system and McKean-Vlasov SDEs.

The aim of this paper is to investigate the well-posedness and propagation of chaos of one-dimensional path dependent McKean-Vlasov SDEs with $\alpha$($\alpha\geq \frac{1}{2}$)-H\"{o}lder continuous diffusion. With respect to the well-posedness, we do not assume that the diffusion is elliptic.

Throughout the paper, fix a constant $r> 0$. Let $\C= C([-r,0];\mathbb{R})$. For any $f\in C([-r,\infty);\mathbb{R})$, $t\geq 0$, define $f_t \in \C$ as $f_t(s)=f(t+s), s\in [-r,0]$, which is called the segment process. Let $\scr P(\R)$ be the set of all probability measures in $\R$ equipped with the weak topology.
Define
$$\scr P_1(\R) = \big\{\mu\in \scr P(\R): \mu(|\cdot|)<\infty\big\}.$$    It is well known that
$\scr P_1(\R)$ is a Polish space under the Wasserstein distance
$$\W_1(\mu,\nu):= \inf_{\pi\in \mathbf{C}(\mu,\nu)} \bigg(\int_{\R\times\R} |x-y| \pi(\d x,\d y)\bigg),\ \ \mu,\nu\in \scr P_{1}(\R),$$ where $\mathbf{C}(\mu,\nu)$ is the set of all couplings of $\mu$ and $\nu$. By the adjoint formula, it holds
$$\W_1(\mu,\nu)=\sup_{\|f\|_{\mathrm{Lip}}\leq 1}|\mu(f)-\nu(f)|,$$
where $$\|f\|_{\mathrm{Lip}}:=\sup_{x\neq y}\frac{|f(x)-f(y)|}{|x-y|}.$$
% Let $\scr P(\C^{n})$ be the set of all probability measures in $\C^{n}$ equipped with the weak topology.
%For $\theta\geq 1$, define
%$$\scr P_\theta(\C^n) = \big\{\mu\in \scr P(\C^n): \mu(\|\cdot\|_\infty^\theta)<\infty\big\}.$$    It is well known that
%$\scr P_\theta(\C^n)$ is a Polish space under the Wasserstein distance
%$$\W_\theta(\mu,\nu):= \inf_{\pi\in \mathbf{C}(\mu,\nu)} \bigg(\int_{\C^n\times\C^n} \|\xi-\eta\|_\infty^\theta \pi(\d \xi,\d \eta)\bigg)^{\ff 1 {\theta}},\ \ \mu,\nu\in \scr P_{\theta}(\C^n),$$ where $\mathbf{C}(\mu,\nu)$ is the set of all couplings of $\mu$ and $\nu$.
Recall that for two probability measures $\mu,\nu$ on   some measurable space $(E,\scr E)$, the entropy and total variation distance are defined as follows:
$$\Ent(\nu|\mu):= \beg{cases} \int_E (\log \ff{\d\nu}{\d\mu})\d\nu, \ &\text{if}\ \nu\ \text{ is\ absolutely\ continuous\ with\ respect\ to}\ \mu,\\
 \infty,\ &\text{otherwise,}\end{cases}$$ and
$$\|\mu-\nu\|_{var} := \sup_{|f|\leq 1}|\mu(f)-\nu(f)|.$$ By Pinsker's inequality (see \cite{Pin}),
\beq\label{ETX} \|\mu-\nu\|_{var}^2\le 2 \Ent(\nu|\mu),\ \ \mu,\nu\in \scr P(E),\end{equation}
here $\scr P(E)$ denotes all probability measures on $(E,\scr E)$.

Let $W=(W(t))_{t\ge0}$ be a one-dimensional standard Brownian motion  on a complete filtration probability space $(\OO,\F,(\F_t)_{t\ge0},\P)$. Consider the following one-dimensional path dependent McKean-Vlasov SDE:
\beq\label{E1}
\d X(t)= b(t,X(t),\L_{X(t)})\d t+B(t,X_t, \L_{X(t)})\d t +\si(t,X(t))\d W(t),
\end{equation}
where $b:[0,\infty)\times\R\times\scr P(\R)\rightarrow\R$, $B:[0,\infty)\times\C\times \scr P(\R)\rightarrow\R$,
$\si:[0,\infty)\times\R\rightarrow\R$ are measurable and $X_0$ is an $\F_0$-measurable $\C$-valued random variable.
Define the uniform norm $\|\xi\|_\infty :=\sup_{s\in[-r,0]} |\xi(s)|, \xi\in\C.$

 \begin{defn}\label{def1} A continuous process $(X(t))_{t\geq
-r}$ on $\mathbb{R}$ is called a strong solution of \eqref{E1} in $\scr P_1(\R)$, if for any $t\geq 0$, $X(t)$ is $\F_t$-measurable,
$\E\|X_t\|_\infty<\infty$,
and $\P$-a.s.
$$X(t)=X(0)+\int_0^t(b(s,X(s),\L_{X(s)})+B(s,X_s,\L_{X(s)}))\d s+\int_0^t\si(s,X(s))\d W(s), ~~~~t\ge0.$$
\end{defn}
Throughout the paper, we fix $T>0$, and consider the solution on $[0,T]$.

The remainder of this paper is organized  as follows: In Section
2, the strong well-posedness of path dependent classical SDEs  is addressed by
Yamada-Watanabe's approximation; In Section 3, the well-posedness and quantitative propagation of chaos for path dependent McKean-Vlasov SDEs are investigated.

\section{Multi-dimensional Path Dependent Classical SDEs with H\"{o}lder Continuous Diffusion}
In this section, we study the well-posedness of multi-dimensional path dependent classical SDEs with H\"{o}lder continuous diffusion. The main tool is the Yamada-Watanabe approximation, see \cite{IW}.
Before going on, we state a time nonhomogeneous version of \cite[Theorem 2.3]{RS}. More precisely, consider path dependent SDE on $\R^m$:
\begin{align}\label{dEd}
\d X(t)=f(t,X_t)\d t+g(t,X_t)\d \bar{W}(t), \ \ X_0=\xi\in\C^m,
\end{align}
here $f:[0,T]\times \C^m\to\R^m$, $g:[0,T]\times \C^m\to\R^m\otimes\R^d$ and $\bar{W}(t)$ is a $d$-dimensional standard Brownian motion on a complete filtration probability space $(\bar{\Omega},\bar{\F},\{\bar{\F}_t\}_{t\in[0,T]},\bar{\P})$.
\begin{prp}\label{exs} Assume that $f$ and $g$ are bounded on bounded sets. Suppose that for any $t\in[0,T]$, $f(t,\cdot)$ and $g(t,\cdot)$ are continuous in $(\C^m,\|\cdot\|_\infty)$ and there exists a constant $K\in\R$ such that
\begin{align*}
&2\<f(t,\xi)-f(t,\eta),\xi(0)-\eta(0)\>+\|g(t,\xi)-g(t,\eta)\|_{HS}^2\leq K\|\xi-\eta\|_\infty^2, \ \ t\in[0,T], \xi,\eta\in\C^m.
\end{align*}
Then \eqref{dEd} has a unique non-explosive strong solution on $[0,T]$.
\end{prp}
Since the proof of Proposition \ref{exs} is completely the same with that of \cite[Theorem 2.3]{RS}, we omit it here.

The following stochastic Gronwall lemma comes from \cite[Lemma 5.2]{RS}, which is crucial in the proof of the main result of this section.
\begin{lem}\label{Gro}
Let $Z$ be a continuous adapted non-negative stochastic process which satisfies the inequality
$$Z(t)\leq  K\int_0^t \sup_{u\in[0,s]}Z(u)\d s + M(t) + C,\ \ t\geq 0,$$
where $C\geq 0$, $K>0$ and $M$ is a continuous local martingale with $M(0) = 0$. Then for any $p\in(0,1)$,
there exist finite constants $c_1(p), c_2(p)$ (not depending on $K, C, T$ and $M$) such that
$$\E(\sup_{s\in[0,t]}|Z(s)|^p) \leq  C^pc_1(p)\e^{c_2(p)Kt},\ \ t\geq 0.$$
\end{lem}
For $x\in\R^d$, we denote $x_i$ as the $i$-th component of $x$, that is $x=(x_1,x_2,\cdots, x_d)$.
Consider
\begin{align}\label{MEQ}
\d X(t)=F(t,X(t))\d t+H(t,X_t)\d t+G(t,X(t))\d \bar{W}(t), \ \ X_0=\xi\in\C^d,
\end{align}
where $F=(F_1,F_2,\cdots, F_d):[0,T]\times \R^d\to\R^d, H=(H_1,H_2,\cdots, H_d):[0,T]\times \C^d\to\R^d$, $G:[0,T]\times \R^d\to\R^d\otimes\R^d$. We make the following assumption.
\begin{enumerate}
\item[{\bf (A1)}] $F$ is locally bounded in $[0,T]\times \R^d$. For any $t\in[0,T]$, $F(t,\cdot)$ is continuous,
and there exists a constant $K_1\in\R$ such that
\begin{equation*}
 (F_i(t,x)-F_i(t,y))\mathrm{sgn}(x_i-y_i)\le K_1|x-y|,\ \ t\in[0,T],x,y\in\R^d, 1\leq i\leq d.
\end{equation*}
\item[{\bf (A2)}] There exist $m$ real valued functions $(G_i)_{1\leq i\leq d}$ on $[0,T]\times \R$ such that $$G(t,x)=\mathrm{diag}(G_1(t,x_1), G_2(t,x_2),\cdots, G_1(t,x_d)), \ \ x=(x_1,x_2,\cdots,x_d)\in\R^d,t\in[0,T].$$
    Moreover, there exist constants $(\alpha_i)_{1\leq i\leq d}\subset [\frac{1}{2},1]$ and
$K_2\geq0$ such that
$$ |G_i(t,z)-G_i(t,\bar{z})|\le
K_2|z-\bar{z}|^{\alpha_i}, \ \ |G_i(t,0)|\leq K_2,\ \ z,\bar{z}\in\R,t\in[0,T], 1\leq i\leq d. $$
\item[{\bf (A3)}] There exists a constant $K_3\geq 0$ such that
\begin{equation*}
|H(t,\xi)-H(t,\eta)|\le K_3\|\xi-\eta\|_{\infty},\ \ |H(t,0)|\leq K_3,\ \ \xi,\eta\in\C^d, t\in[0,T].
\end{equation*}
\end{enumerate}
Now, we provide the main result in this section.
\begin{thm} \label{Existence} Assume ${\bf (A1)}$-${\bf (A3)}$. Then for any $\xi\in \C^d$, \eqref{MEQ}
has a unique strong solution $(X^\xi(t))_{t\in[-r,T]}$ with initial
value $X_0^\xi=\xi$ and for any $q\in(0,2]$, there exists a constant $C(T,q)>0$ such that
\begin{equation}\label{eq3}
\E\sup_{t\in[-r,T]}|X^\xi(t)|^q\le C(T,q)(1+\|\xi\|_\infty^q).
\end{equation}
Moreover, for any $p\in(0,1)$,
\begin{equation}\label{EQA}
\E\sup_{t\in[-r,T]}|X^\xi(t)-X^\eta(t)|^p\le C(p,T)\|\xi-\eta\|_\infty^p
\end{equation}
for some constant $C(p,T)>0$.
%Consequently, for any $\beta\in(0,1)$, it holds
%\begin{equation}\label{EQB}
%\E\sup_{t\in[-r,T]}|X^\xi(t)-X^\eta(t)|^2\le C\|\xi-\eta\|_\infty^2+C(\beta)\frac{2-2\alpha}{2-\beta}\|\xi-\eta\|_\infty^{\beta}.
%\end{equation}
\end{thm}

\begin{proof} {\bf Step 1. Existence of the strong solution.}

For $\vv\in(0,1),$ note $\int_{\vv/\e^{\frac{1}{\varepsilon}}}^\vv\ff{\varepsilon}{x}\d
x=1$, so there exists a continuous function
$\psi_{\vv}:[0,\infty)\to[0,\infty)$ with support $[\vv/\e^{\frac{1}{\varepsilon}},\vv]$ such that
\begin{equation*}
0\le\psi_{\vv}(x)\le \ff{2\varepsilon}{x},~~~x\in[\vv/\e^{\frac{1}{\varepsilon}},\vv],  ~~\int^\vv_{\vv/\e^{\frac{1}{\vv}}} \psi_{\vv}(u)\d u=1.
\end{equation*}
Let
\begin{equation*}
V_{\vv}(x):=\int_0^{|x|}\int_0^y\psi_{\vv}(z)\d z\d y,\ \ x\in\R.
\end{equation*}
Then $V_{\vv}\in C^2$,
\begin{equation}\label{R1}
|x|-\vv\le V_{\vv}(x)\le |x|,~~ 0\leq \mbox{sgn}(x) V_{\vv}'(x)
\leq 1, ~~x\in\R,
\end{equation}
and
\begin{equation}\label{R2}
 0\le V_{\vv}''(x)\le
\ff{2\varepsilon}{|x|}{\bf1}_{[\vv/\e^{\frac{1}{\varepsilon}},\vv]}(|x|),~~~~x\in\R.
\end{equation}
Let $\rho\in C_0^\infty(\R)$ with $\rho\geq 0$ and $\int_{\R}\rho(x)\d x=1$ be supported in $[-1,1]$. For any $n\geq 1$, define $\rho^n(x)=n\rho(nx), x\in\R$ and let
\begin{align*}G^n_i(t,\cdot)=G_i(t,\cdot)\ast\rho^n,\ \ 1\leq i\leq d, t\in[0,T]
\end{align*}
and
$$G^n(t,x)=\mathrm{diag}(G^n_1(t,x_1), G^n_2(t,x_2),\cdots, G^n_d(t,x_d)), \ \ x=(x_1,x_2,\cdots,x_d)\in\R^d,t\in[0,T].$$
{\bf (A2)} implies that
\begin{equation}\label{uni}\begin{split}&\lim_{n\to\infty}\sup_{t\in[0,T],x\in\R}|G_i^n(t,x)-G_i(t,x)| \\ &\leq \lim_{n\to\infty}\sup_{t\in[0,T],x\in\R}\int_{\R}\rho^n(y)|G_i(t,x-y)-G_i(t,x)|\d y\\
&\leq K_2\lim_{n\to\infty}\int_{\R}|y|^{\alpha_i}\rho^n(y)\d y\\
&=K_2\lim_{n\to\infty}\int_{\R}\frac{1}{n^{\alpha_i}}|x|^{\alpha_i}\rho(x)\d x=0,\ \ 1\leq i\leq d
\end{split}\end{equation}
and for any $n\geq 1$, there exists a constant $K_n\geq 0$ such that
\begin{align}\label{sin}&|G^n(t,x)-G^n(t,y)|\leq K_n|x-y|,\ \ |G^n(t,0)|\leq 2\sqrt{d}K_2,\ \ t\in[0,T],x,y\in\R^d,
\end{align}
where for the second inequality, it is sufficient to note that
$$|G_i^n(t,0)|\leq\int_{\R}|G_i(t,-y)|n\rho(ny)\d y\leq K_2\int_{\R}|y|^{\alpha_i}n\rho(ny)\d y+K_2\leq 2K_2,\ \ 1\leq i\leq d. $$
According to Proposition \ref{exs}, it follows from {\bf(A1)}, {\bf(A3)} and \eqref{sin} that for any $n\geq 1$, the SDE
\begin{equation}\label{SOS}
\d X^n(t)= F(t,X^n(t))\d t+H(t,X^n_{t})\d t +G^n(t,X^n(t))\d \bar{W}(t),\ \ X^n_0=\xi
\end{equation}
has a unique non-explosive strong solution. Moreover, in view of the second inequality of \eqref{sin} and
\begin{align}\label{GIN}|G^n_i(t,z)-G^n_i(t,\bar{z})|\leq K_2|z-\bar{z}|^{\alpha_i}, \ \ t\in[0,T],z,\bar{z}\in\R,1\leq i\leq d,n\geq 1,\end{align}
this together with {\bf(A1)} and {\bf(A3)} implies that there exists a constant $C(T)>0$ such that
\begin{align}\label{sun}\sup_{n\geq 1}\E\sup_{t\in[-r,T]}|X^n(t)|^2\leq C(T)(1+\|\xi\|_\infty^2).
\end{align}
In fact, by {\bf(A1)}, {\bf(A3)}, \eqref{GIN} and the second inequality in \eqref{sin}, it holds
\begin{equation}\begin{split}\label{LIN}&\<F(t,x),x\>\leq C(T)(1+|x|^2),\ \ |H(t,\xi)|\leq C(T)(1+\|\xi\|_\infty),\\
&|G^n(t,x)|\leq C(T)(1+|x|),\ \ x\in\R^d,t\in[0,T], \xi\in\C^d, n\geq 1.
\end{split}\end{equation}
For each integer  $N\ge1$, define $\tau^n_N=\inf\big\{t\in[0,T]: |X^{n}(t)|\ge N\big\}$ and $\inf{\emptyset}=\infty$ by convention.
 By It\^{o}'s formula, we derive from \eqref{LIN} that
\begin{equation*}
\begin{split}
|X^{n}(t\wedge\tau_N^n)|^2
&\le \xi(0)^2+C+C\int_0^{t\wedge\tau_N^n}\sup_{u\in[-r,s]}|X^{n}(u)|^2\d s\\
&+2\int_0^{t\wedge \tau_N^n}\<X^{n}(s),G^n(s,X^n(s))\d \bar{W}(s)\>\\
\end{split}
\end{equation*}
for some constant $C>0$. Applying BDG's inequality, \eqref{LIN} and Gr\"{o}nwall's ineuqality, it is standard to derive \eqref{sun}.
For any $1\leq i\leq d$, let $X^{n,i}$ be the $i$-th component of $X^{n}$.
For any $m,n\geq 1$, it follows from It\^{o}'s formula that
\begin{align*}V_\varepsilon(X^{m,i}(t)-X^{n,i}(t))&=\int_0^tV_\varepsilon'(X^{m,i}(s)-X^{n,i}(s))(F_i(s,X^{m}(s)) -F_i(s,X^{n}(s)))\d s\\
&+\int_0^tV_\varepsilon'(X^{m,i}(s)-X^{n,i}(s))(H_i(s,X^m_{s})-H_i(s,X^n_{s}))\d s\\
&+\frac{1}{2}\int_0^tV_\varepsilon''(X^{m,i}(s)-X^{n,i}(s))(G_i^m(s,X^{m,i}(s))-G_i^n(s,X^{n,i}(s)))^2\d s\\
&+\int_0^tV_\varepsilon'(X^{m,i}(s)-X^{n,i}(s))(G_i^m(s,X^{m,i}(s))-G_i^n(s,X^{n,i}(s)))\d \bar{W}^i(s).
\end{align*}
This combined with \eqref{R1} implies that for any $1\leq i\leq d$,
\begin{align*}
&\sum_{i=1}^d|X^{m,i}(t)-X^{n,i}(t)|-d\vv\\
&\leq \int_0^t\sum_{i=1}^dV_\varepsilon'(X^{m,i}(s)-X^{n,i}(s))(F_i(s,X^{m}(s)) -F_i(s,X^{n}(s)))\d s\\
&+\int_0^t\sum_{i=1}^dV_\varepsilon'(X^{m,i}(s)-X^{n,i}(s))(H_i(s,X^m_{s})-H_i(s,X^n_{s}))\d s\\
&+\frac{1}{2}\int_0^t\sum_{i=1}^dV_\varepsilon''(X^{m,i}(s)-X^{n,i}(s))(G_i^m(s,X^{m,i}(s))-G_i^n(s,X^{n,i}(s)))^2\d s\\
&+\int_0^t\sum_{i=1}^dV_\varepsilon'(X^{m,i}(s)-X^{n,i}(s))(G_i^m(s,X^{m,i}(s))-G_i^n(s,X^{n,i}(s)))\d \bar{W}^i(s)\\
&=:J^{m,n}_1+J^{m,n}_2+J^{m,n}_3+J^{m,n}_4,\ \ t\in[0,T].
\end{align*}
By \eqref{R1} and {\bf(A1)}, we conclude that
\begin{align*}J^{m,n}_1&=\int_0^t\sum_{i=1}^d\bigg[V_\varepsilon'(X^{m,i}(s)-X^{n,i}(s))\mathrm{sgn}(X^{m,i}(s)-X^{n,i}(s))\\
&\qquad\qquad\quad\times (F_i(s,X^{m}(s)) -F_i(s,X^{n}(s)))\mathrm{sgn}(X^{m,i}(s)-X^{n,i}(s))\bigg]\d s\\
&\leq K_1^+ d\int_0^t|X^{m}(s)-X^{n}(s)|\d s,\ \ t\in[0,T].
\end{align*}
\eqref{R1} and {\bf(A3)} yield that
$$J^{m,n}_2\leq K_3d\int_0^t\|X^m_s-X^n_s\|_\infty\d s,\ \ t\in[0,T].$$
Note that \eqref{R2} and \eqref{GIN} derive
\begin{equation*}\begin{split}
J^{m,n}_3
&\leq 2K_2^2\int_0^t\sum_{i=1}^d\varepsilon^{2\alpha_i}\d s +2\int_0^t\sum_{i=1}^d\e^{\frac{1}{\varepsilon}}(G_i^m(s,X^{n,i}(s))-G_i^n(s,X^{n,i}(s)))^2\d s\\
&\leq 2tK_2^2\sum_{i=1}^d\varepsilon^{2\alpha_i} +2t\e^{\frac{1}{\varepsilon}}\sum_{i=1}^d\sup_{s\in[0,t], z\in\R}|G_i^m(s,z)-G_i^n(s,z)|^2\\
&=:H^{m,n,\varepsilon}(t),\ \ t\in[0,T].
\end{split}\end{equation*}
So, \eqref{R1} and Lemma \ref{Gro} imply that for any $p\in(0,1)$, there exist constants $c_1(p), c_2(p)>0$ such that
\begin{align*}\E\sup_{t\in[0,T]}|X^m(t)-X^n(t)|^p\leq c_1(p)\e^{c_2(p)(K_1^++K_3)dT}(H^{m,n,\varepsilon}(T)+d\varepsilon)^p.
\end{align*}
Thanks to \eqref{uni}, we derive
\begin{align*}
\limsup_{m,n\to\infty}|H^{m,n,\varepsilon}(T)|\leq 2TK_2^2\sum_{i=1}^d\varepsilon^{2\alpha_i}.
\end{align*}
So, for $p\in(0,1)$, it holds
\begin{align*}\limsup_{m,n\to\infty}\E\sup_{t\in[0,T]}|X^m(t)-X^n(t)|^p\leq c_1(p)\e^{c_2(p)(K_1^++K_3)dT}(2TK_2^2\sum_{i=1}^d\varepsilon^{2\alpha_i}+d\varepsilon)^p.
\end{align*}
Letting $\varepsilon\to0$, we conclude that for any $p\in(0,1)$,
\begin{align*}\limsup_{m,n\to\infty}\E\sup_{t\in[0,T]}|X^m(t)-X^n(t)|^p=0.
\end{align*}
So, there exists a continuous stochastic process $\{\bar{X}(t)\}_{t\in[-r,T]}$ satisfying $\bar{X}_0=\xi$ and
\begin{align}\label{lis}\limsup_{n\to\infty}\E\sup_{t\in[0,T]}|X^n(t)-\bar{X}(t)|^p=0.
\end{align}
This yields that there exists a subsequence $\{n_k\}_{k\geq 1}$ such that $\P$-a.s.
\begin{align}\label{ALS}\lim_{k\to\infty}\sup_{t\in[0,T]}|X^{n_k}(t)-\bar{X}(t)|=0,,\ \ \sup_{k\geq 1}\sup_{t\in[-r,T]}(|X^{n_k}(t)|+|\bar{X}(t)|)<\infty.
\end{align}
Moreover, \eqref{sun} and Fatou's Lemma imply
$$\E\sup_{t\in[-r,T]}|\bar{X}(t)|^2\leq C(T)(1+\|\xi\|_\infty^2).$$
So, by the local boundedness of $F,H$, the continuity of $F(s,\cdot), H(s,\cdot)$, \eqref{ALS} and the dominated convergence theorem, we conclude that $\P$-a.s.
$$\lim_{k\to\infty}\sup_{t\in[0,T]}\left|\int_0^t(F(s,X^{n_k}(s))+H(s,X^{n_k}_s))\d s-\int_0^t(F(s,\bar{X}(s))+H(s,\bar{X}_s))\d s\right|=0.$$
Moreover, by Markov's inequality, BDG's inequality, \eqref{uni}, \eqref{GIN} and \eqref{lis}, for any $\varepsilon>0$ and $p\in(0,1)$, we have
\begin{align*}&\limsup_{k\to\infty}\P\left(\sup_{t\in[0,T]}\left|\int_0^t[G^{n_k}(s,X^{n_k}(s))-G(s,\bar{X}(s))]\d \bar{W}(s)\right|\geq \varepsilon\right)\\
&\leq\limsup_{k\to\infty}\frac{1}{\varepsilon^p}\E\sup_{t\in[0,T]}\left|\int_0^t[G^{n_k}(s,X^{n_k}(s))-G(s,\bar{X}(s))]\d \bar{W}(s)\right|^p\\
&\leq c(p)\limsup_{k\to\infty}\frac{1}{\varepsilon^p}\E\left(\int_0^T\sum_{i=1}^d[G_i^{n_k}(s,X^{n_k,i}(s))-G_i(s,\bar{X}^i(s))]^2\d s\right)^{\frac{p}{2}}\\
&\leq c(p)T^{\frac{p}{2}}\frac{1}{\varepsilon^p}\limsup_{k\to\infty}\left(\sum_{i=1}^d\sup_{s\in[0,T],x\in\R}|G_i^{n_k}(s,x)-G_i(s,x)|^2\right)^{\frac{p}{2}}\\
&+c(p)T^{\frac{p}{2}}K_2^p\frac{1}{\varepsilon^p}\limsup_{k\to\infty} \E \left(\sup_{s\in[0,T]}\sum_{i=1}^d|X^{n_k,i}(s)-\bar{X}^i(s)|^{2\alpha_i}\right)^{\frac{p}{2}}=0.
\end{align*}
Therefore, replacing $n$ by $n_k$ in \eqref{SOS} and letting $k\to\infty$, it holds $\P$-a.s.
\begin{equation*}
\bar{X}(t)= \int_0^tF(s,\bar{X}(s))\d s+\int_0^tH(s,\bar{X}_{s})\d s +\int_0^tG(s,\bar{X}(s))\d \bar{W}(s),\ \ t\in[0,T].
\end{equation*}
This means that $\{\bar{X}(t)\}_{t\in[-r,T]}$ is a strong solution to \eqref{MEQ}.

{\bf Step 2. Uniqueness of the strong solution.}

Let $X^\xi(t)$ be the solution to \eqref{MEQ} with initial value $\xi\in\C^d$. By the same argument to derive \eqref{sun}, we obtain
 $$\E\sup_{t\in[-r,T]}|X^{\xi}(t)|^2
\le C(T)(1+\|\xi\|_\infty^2).$$
So, Jensen's inequality implies \eqref{eq3}. For any $1\leq i\leq d$, let $X^{\xi,i}$ be the $i$-th component of $X^{\xi}$.
Applying Ito's formula, for any $1\leq i\leq d$, we have
\begin{equation*}
\begin{split}
V_\vv(X^{\xi,i}(t)-X^{\eta,i}(t))&=V_\vv(\xi^i(0)-\eta^i(0))\\%-\ll \int_0^t\e^{-\ll s}V_\vv(Z_s^{(k+1)})\d s\\&\quad+
&+\int_0^t
V'_\vv(X^{\xi,i}(s)-X^{\eta,i}(s))\big\{F_i(s,X^\xi(s))-F_i(s,X^\eta(s))\big\} \d s\\
&+\int_0^t
V'_\vv(X^{\xi,i}(s)-X^{\eta,i}(s))\big\{H_i(s,X^\xi_s)-H_i(s,X^\eta_s)\big\} \d s\\
&+\frac{1}{2}\int_0^t
V''_\vv(X^{\xi,i}(s)-X^{\eta,i}(s))\big\{G_i(s,X^{\xi,i}(s))-G_i(s,X^{\eta,i}(s))\big\}^2 \d s\\
&+\int_0^t V'_\vv(X^{\xi,i}(s)-X^{\eta,i}(s)) \big\{G_i(s,X^{\xi,i}(s))-G_i(s,X^{\eta,i}(s)) \big\}\d \bar{W}^i(s).
\end{split}
\end{equation*}
By {\bf (A1)}-{\bf(A3)}, \eqref{R1} and \eqref{R2}, it holds
\begin{equation*}
\begin{split}
|X^\xi(t)-X^\eta(t)|&\leq d\varepsilon+K_2^2T\sum_{i=1}^d\varepsilon^{2\alpha_i}+C(T)\|\xi-\eta\|_\infty\\
&+C\int_0^t
\sup_{u\in[0,s]}|X^\xi(u)-X^\eta(u)|\d s+M_t, \ \ t\in[0,T]
\end{split}
\end{equation*}
for a martingale $M_t$ and some constants $C, C(T)>0$.
Then for any $p\in(0,1)$, applying Lemma \ref{Gro}, we get
\begin{align*}\E\sup_{t\in[0,T]}|X^\xi(t)-X^\eta(t)|^p\leq c_1(p)\e^{c_2(p)T}\left(d\varepsilon+K_2^2T\sum_{i=1}^d\varepsilon^{2\alpha_i}+C(T)\|\xi-\eta\|_\infty\right)^p.
\end{align*}
Letting $\vv\to0$, we derive \eqref{EQA}, which yields the uniqueness of the strong solution of \eqref{MEQ}.
\end{proof}
\section{Path Dependent McKean-Vlasov SDEs with H\"{o}lder Continuous Diffusion}
Throughout this section, we make the following assumption.
\begin{enumerate}
\item[{\bf(H)}] Assume that the following conditions hold.
\begin{enumerate}
\item[({\bf Hb})] $b$ is locally bounded in $[0,T]\times \R\times\scr P_1(\R)$. For any $t\in[0,T]$, $b(t,\cdot,\cdot)$ is continuous in $\R\times \scr P_1(\R)$,
and there exists a constant $K_b\in\R$ such that for
$x,y\in\R$ and $\mu,\nu\in\scr P_1(\R)$,
\begin{equation*}
 [b(t,x,\mu)-b(t,y,\nu)]\mathrm{sgn}(x-y)\le K_b(\W_1(\mu,\nu)+|x-y|),\ \ t\in[0,T],
\end{equation*}
where $\mbox{sgn}(\cdot)$ means the sign function.
\item[({\bf H$\sigma$})] There exist constants
$K_\sigma\geq0$ and $\aa\in[\ff{1}{2},1]$ such that  $$ |\si(t,x)-\si(t,y)|\le
K_\sigma|x-y|^{\alpha}, \ \ |\sigma(t,0)|\leq K_\sigma,\ \ x,y\in\R,t\in[0,T] $$

\item[({\bf HB})] There exists a constant $K_B\geq 0$ and a probability measure $m$ on $[-r,0]$ such that for any $\xi,\eta\in\C,\mu,\nu\in\scr P_1(\R),t\in[0,T]$,
\begin{equation*}
|B(t,\xi,\mu)-B(t,\eta,\nu)|\le K_B\big\{\|\xi-\eta\|_{L^1(m)}+\mathbb{W}_1(\mu,\nu)\},\ \ |B(t,0,\delta_0)|\leq K_B,
\end{equation*}
here $\delta_0$ is the Dirac measure at the point $0$.
\end{enumerate}
\end{enumerate}
\subsection{Well-posedness}
\begin{thm} \label{EXI} Assume {\bf (H)}. Then for any   $X_0\in L^1(\OO\to(\C,\|\cdot\|_\infty);\F_0,\P)$, \eqref{E1}
has a unique strong solution $(X(t))_{t\in[-r,T]}$ with initial
value $X_0$ and there exists a constant $C(T)>0$ such that
\begin{equation}\label{FMT}
\E\sup_{t\in[0,T]}\|X_t\|_\infty\le C(T)(1+\E\|X_0\|_\infty).
\end{equation}
Moreover, for two solutions $X(t)$ and $\tilde{X}(t)$,
\begin{equation}\begin{split}\label{DDS}
 &\sup_{t\in[0,T]}\E|X(t)-\tilde{X}(t)|\\
 &\quad\le C(T)\E\left\{|X(0)-\tilde{X}(0)|+K_B\int_{-r}^0 m([-r,u])|X(u)-\tilde{X}(u)|\d u\right\}.
\end{split}\end{equation}
\end{thm}

\begin{proof}
For $\mu\in C([0,T]; \scr P_1(\R)) $, $x\in\R$ and $\xi\in\C$, let $b^\mu(t,x) =b(t,x,\mu_t)$, $B^\mu(t,\xi)=B(t,\xi,\mu_t)$.  Consider
\beq\label{EW1} \d X^\mu(t)=b^\mu(t,X^\mu(t))\d t+B^\mu(t,X^\mu_t)\d t +\si(t,X^\mu(t))\d W(t),\ \  t\in [0,T].\end{equation}
By {\bf (H)} and Theorem \ref{Existence}, \eqref{EW1} is strongly well-posed and let $\Phi_t(\mu)=\L_{X^\mu(t)}, t\in[0,T]$, where $(X^\mu(t))_{t\in[-r,T]}$ solves  \eqref{EW1} with $X^\mu_0\in L^1(\OO\to(\C,\|\cdot\|_\infty);\F_0,\P)$. In view of {\bf(Hb)},
\begin{align}\label{mon}b(t,x,\mu)\mathrm{sgn}(x)\leq C_0(T)(1+|x|+\mu(|\cdot|)), \ \ t\in[0,T]
\end{align}
holds for some constant $C_0(T)>0$.
So, by the similar argument to derive \eqref{sun}, we get
\begin{align}\label{SEM}
\E(\sup_{s\in[-r,t]}|X^\mu(s)|^2|\F_0)\leq C(T)^2\left(1+\|X_0^\mu\|_\infty^2+\int_0^t\mu_s(|\cdot|)^2\d s\right),\ \ t\in[0,T],
\end{align}
which yields
\begin{align}\label{CTS}
\E(\sup_{s\in[-r,t]}|X^\mu(s)|)\leq C(T)\left(1+\E\|X_0^\mu\|_\infty+\left(\int_0^t\mu_s(|\cdot|)^2\d s\right)^{\frac{1}{2}}\right),\ \ t\in[0,T]
\end{align}
for some constant $C(T)\geq 0.$
%For any $\gg_1,\gg_2\in\scr P_2(\R^d),$ we can take
%$\F_0$-measurable random variables $X_0^\mu$ with $ \L_{X_0^\mu}=\gg_1$  and $X_0^\nu$ with $\L_{X_0^\nu}=\gg_2$  such that
% $$\E|X_0^\mu-X_0^\nu|^2 =\W_2(\gg_1,\gg_2)^2.$$
%Below, we aim to show that there exists a constant $c>0$ such that for    $\mu_\cdot,\nu_\cdot\in C([0,T]; \scr P_2(\R^d)) $,
%\begin{align}\label{AA}
%\W_2\big(\Phi_t(\mu),\Phi_t(\nu)\big)^{2}
% \le c \int_0^t\W_\theta(\mu_s,\nu_s)^2\d s£¬~~~~t\in[0,T].
% \end{align}
By It\^o's formula,  it
follows that
\begin{equation}\label{EE}
\begin{split}
V_\vv(X^\mu(t)-X^\nu(t))&=V_\vv(X^\mu(0)-X^\nu(0))\\%-\ll \int_0^t\e^{-\ll s}V_\vv(Z_s^{(k+1)})\d s\\&\quad+
&+\int_0^t
V'_\vv(X^\mu(s)-X^\nu(s))\big\{b(s,X^\mu(s),\mu_s)-b(s,X^\nu(s),\nu_s)\big\} \d s\\
&+\int_0^t
V'_\vv(X^\mu(s)-X^\nu(s))\big\{B(s,X^\mu_s,\mu_s)-B(s,X^\nu_s,\nu_s)\big\} \d s\\
&+\frac{1}{2}\int_0^t
V''_\vv(X^\mu(s)-X^\nu(s))\big\{\si(s,X^\mu(s))-\si(s,X^\nu(s))\big\}^2 \d s\\
&+\int_0^tV'_\vv(X^\mu(s)-X^\nu(s)) \big\{\si(s,X^\mu(s))-\si(s,X^\nu(s)) \big\}\d W(s)\\
&=:
I_{1,\vv}+I_{2,\vv}(t)+I_{3,\vv}(t)+I_{4,\vv}(t)+I_{5,\vv}(t).
\end{split}
\end{equation}
Using \eqref{R1}, we get
\begin{equation*}
\begin{split}
I_{1,\vv}\le |X^\mu(0)-X^\nu(0)|.
\end{split}
\end{equation*}
Moreover, it follows from \eqref{R1} and {\bf (Hb)} that
\begin{equation*}
\begin{split}
I_{2,\vv}(t)
&\le K_b^+\int_0^t
\big\{|X^\mu(s)-X^\nu(s)|+\mathbb{W}_1(\mu_s,\nu_s)\big\}\d s, \ \ t\in[0,T].
\end{split}
\end{equation*}
By \eqref{R1}, {\bf (HB)} and Fubini's theorem, we arrive at
\begin{align*}
&I_{3,\vv}(t) \leq K_B\int_0^t\big\{\|X^\mu_s-X^\nu_s\|_{L^1(m)}+\mathbb{W}_1(\mu_s,\nu_s)\big\}\d s\\
&= K_B\int_{-r}^0\left(\int_\theta^{t+\theta}|X^\mu(u)-X^\nu(u)|\d u\right) m(\d \theta)+K_B\int_0^t\mathbb{W}_1(\mu_s,\nu_s)\d s\\
&\leq K_B\int_{-r}^0\left(\int_0^{t}|X^\mu(u)-X^\nu(u)|\d u\right) m(\d \theta)\\
&+ K_B\int_{-r}^0\left(\int_\theta^{0}|X^\mu(u)-X^\nu(u)|\d u\right) m(\d \theta)+K_B\int_0^t\mathbb{W}_1(\mu_s,\nu_s)\d s\\
&\leq K_B\left(\int_0^{t}|X^\mu(u)-X^\nu(u)|\d u\right)\\
 &+K_B\int_{-r}^0 m([-r,u])|X^\mu(u)-X^\nu(u)|\d u+K_B\int_0^t\mathbb{W}_1(\mu_s,\nu_s)\d s,\ \ t\in[0,T].
\end{align*}
Furthermore, by {\bf (H$\sigma$)}, \eqref{R2} and using $\aa\in[1/2,1]$, we deduce
\begin{equation*}
\begin{split}
I_{4,\vv}(t) \le K_\sigma^{2} T\varepsilon^{2\alpha},\ \ t\in[0,T].
\end{split}
\end{equation*}
%If {\bf(H2)} holds, then
%$$| I_2(t)|\le c_4K^2t\vv.$$
%Moreover, \eqref{eq6} yields
%\begin{equation}\label{qu}
%\<I_3\>(t)\leq c\int_0^t|Z_s^{(k+1)}|^{2\alpha}\d s.
%\end{equation}
In addition, by \eqref{R1}, {\bf (H$\sigma$)} and \eqref{CTS}, we have $\E I_{5,\vv}(t)=0$.
Taking expectation in \eqref{EE}, using \eqref{R1} and letting $\vv\downarrow0$, there exists a constant $C>0$ such that
\begin{equation}\label{W11}\begin{split}
 \E|X^\mu(t)-X^\nu(t)|&\le \E|X^\mu(0)-X^\nu(0)|+K_B\int_{-r}^0 m([-r,u])\E|X^\mu(u)-X^\nu(u)|\d u\\
 &+C\int_0^t\E|X^\mu(s)-X^\nu(s)| \d
s+(K_b^++K_B)\int_0^t\mathbb{W}_1(\mu_s,\nu_s)\d s.
\end{split}\end{equation}
It follows from Gr\"{o}nwall's inequality that
\begin{equation*}\begin{split}
\E|X^\mu(t)-X^\nu(t)|&\le \e^{Ct}\left\{\E|X^\mu(0)-X^\nu(0)|+K_B\int_{-r}^0 m([-r,u])\E|X^\mu(u)-X^\nu(u)|\d u\right\}\\
 &+C(T)\int_0^t\mathbb{W}_1(\mu_s,\nu_s)\d s.
\end{split}\end{equation*}
So, when $X^\mu_0=X^\nu_0$, for $\lambda=2C(T)$, we get
\begin{equation*}
\begin{split}
 \sup_{t\in[0,T]}\e^{-\lambda t}\W_1(\Phi_t(\mu),\Phi_t(\nu))\leq\frac{1}{2}\sup_{t\in[0,T]}\e^{-\lambda t}\mathbb{W}_1(\mu_t,\nu_t).
\end{split}
\end{equation*}
Set
$$E_{\lambda}:= \big\{\mu\in C([0,T]; \scr P_1(\R)):\mu_0=\L_{X^\mu(0)}\big\}$$ equipped with the complete metric
$$ \rr(\mu,\nu):= \sup_{t\in[0,T]} \e^{-\lambda t}\W_1(\mu_t,\nu_t),\ \ \mu,\nu\in E_\lambda.$$
Then $\Phi$ is strictly contractive in $E_{\lambda}$. Consequently, the Banach fixed point theorem together with the definition of $\Phi$ implies that there exists a unique $\mu\in E_{\lambda}$ such that
 $$\Phi_t(\mu)=\mu_t=\L_{X^\mu(t)}, ~~~t\in[0,T].$$
% Whence, the unique strong solution of \eqref{EW1} is a strong solution of \eqref{C2}. On the other hand, since any strong solution of \eqref{C2} is a strong solution to \eqref{EW1}, then the uniqueness of \eqref{EW1} implies the one of \eqref{C2}.
Finally, taking $\mu_t=\L_{X^\mu(t)}$ in \eqref{CTS}, \eqref{FMT} follows from Gr\"{o}nwall's inequality. Similarly, taking $\mu_t=\L_{X(t)}, \nu_t=\L_{\tilde{X}(t)}$, $X^\mu(t)=X(t)$, $X^\nu(t)=\tilde{X}(t)$ in \eqref{W11},
\eqref{DDS} holds by Gr\"{o}nwall's inequality.

\end{proof}
\subsection{Propagation of Chaos}
Let
$N\ge1$ be an integer and $(X_0^i,W^i(t))_{1\le i\le N}$ be i.i.d.\,copies of $(X_0,W(t))$ with $\F_0$-measurable $\C$-valued random variable $X_0$. Consider
\begin{equation*}
\d X^i(t)=b(t,X^i(t),\L_{X^i(t)})\d t+B(t,X_t^i,\L_{X^i(t)})\d t+\si(t,X^i(t))\d
W^i(t),\ \ 1\leq i\leq N.
\end{equation*}
 Let
\begin{equation}\label{H1}
\tt\mu^N_t =\ff{1}{N}\sum_{j=1}^N\dd_{X^j(t)}.
\end{equation}
Consider the stochastic $N$-interacting particle
system:
\begin{equation}\label{eq4}
\d X^{i,N}(t)=b(t,X^{i,N}(t),\hat\mu_t^N)\d t+B(t,X_t^i,\hat\mu_t^N)\d t+\si(t,X^{i,N}(t))\d
W^i(t),\ \ X_0^{i,N}=X_0^i,
\end{equation}
where  $\hat\mu_t^N$ is the empirical distribution corresponding
to  $X^{1,N}(t),\cdots,X^{N,N}(t)$, i.e.
\begin{equation*}
 \hat\mu_t^N :=\ff{1}{N}\sum_{j=1}^N\dd_{X^{j,N}(t)}.
 \end{equation*}
Applying Theorem \ref{Existence}, the well-posedness of the stochastic  $N$-interacting
 particle system \eqref{eq4} can be proved in the following lemma.

\begin{lem}\label{lem}
 Assume {\bf(H)} and $X_0^i\in L^1(\Omega\to(\C,\|\cdot\|_\infty);\F_0,\P), 1\leq i\leq N$. Then, for each $N\ge1$, \eqref{eq4} admits a unique strong solution $\{(X^{i,N}(t))_{t\in[-r,T]}\}_{1\leq i\leq N}$ and
\begin{equation}\label{ets}\E \sup_{t\in[-r,T]}|X^{i,N}(t)|\le C(T)(1+\E\|X_0^i\|_{\infty}),\ \ 1\leq i\leq N\end{equation}
holds for some constant $C(T)>0$.
\end{lem}
\begin{proof}
For $x:=(x_1,x_2,\cdots,x_N)^*\in\R^N$, $\xi:=(\xi_1,\xi_2,\cdots,\xi_N)^*\in\C^N$, set $\tt\mu^N_x=\frac{1}{N}\sum_{i=1}^N\delta_{x_i}$ and
\begin{equation*}
\begin{split}
&\hat b(t,x):=\big(b(t,x_1,\tt\mu^N_x),\cdots,b(t,x_N,\tt\mu^N_x)\big)^*,~~~\hat B(t,\xi):=\big(B(t,\xi_1,\tt\mu^N_{\xi(0)}),\cdots,B(t,\xi_N,\tt\mu^N_{\xi(0)})\big)^*,\\
& \hat\si(t,x):=\mbox{diag}\big(\si(t,x_1),\cdots,\si(t,x_N)\big),~~~~\hat W(t):=\big(W^1(t),\cdots,W^N(t)\big)^*,\ \ t\in[0,T].
\end{split}
\end{equation*}
Then it is clear that $(\hat W(t))_{t\in[0,T]}$ is an $N$-dimensional Brownian
motion and \eqref{eq4} can be reformulated as
\begin{equation}\label{B1}
\d \hat{X}(t)=\hat b(t,\hat{X}(t))\d t+\hat{B}(t,\hat{X}_t)\d t+\hat\si(t,\hat{X}(t))\d \hat W(t),\ \ \hat{X}_0=(X_0^1,X_0^2,\cdots,X_0^N)^\ast.
\end{equation}
Note that
\begin{equation}\label{W3}
\W_1\left(\ff{1}{N}\sum_{i=1}^N\dd_{x_i},\ff{1}{N}\sum_{i=1}^N\dd_{\tilde{x}_i}\right)\le
\ff{1}{N}\sum_{i=1}^N|x_i-\tilde{x}_i|,\ \ x_i,\tilde{x}_i\in\R, 1\leq i\leq N.
\end{equation}
It is not difficult to see from {\bf (Hb)}, {\bf(HB)} and \eqref{W3} that $\hat{b}$ is locally bounded in $[0,T]\times \R^N$, for any $t\in[0,T]$, $\hat{b}(t,\cdot)$ is continuous,
\begin{equation}\label{ccb}\begin{split}
(\hat{b}_i(t,x)-\hat{b}_i(t,y))\mathrm{sgn}( x_i-y_i)
&=(b(t,x_i,\tt\mu^N_x)-b(t,y_i,\tt\mu^N_y))\mathrm{sgn}(x_i-y_i)\\
&\leq K_b(|x_i-y_i|+\W_1(\tt\mu^N_x,\tt\mu^N_y))\\
&\leq K_b^+(|x_i-y_i|+\frac{1}{N}\sum_{i=1}^N|x_i-y_i|)\\
&=K_b^+(1+N^{-\frac{1}{2}})|x-y|,\ \ x,y\in\R^N,1\leq i\leq N,
\end{split}\end{equation}
and
\begin{equation}\label{ccB}\begin{split}
|\hat{B}(t,\xi)-\hat{B}(t,\eta)|^2&\leq \sum_{i=1}^N|B(t,\xi_i,\tt\mu^N_{\xi(0)})-B(t,\eta_i,\tt\mu^N_{\eta(0)})|^2\\
&\leq 2K_B^2\sum_{i=1}^N(\|\xi_i-\eta_i\|_\infty^2+\W_1(\tt\mu^N_{\xi(0)},\tt\mu^N_{\eta(0)})^2)\\
&\leq 2K_B^2\sum_{i=1}^N(\|\xi_i-\eta_i\|_\infty^2+|\xi_i(0)-\eta_i(0)|^2)\\
&\leq 4K_B^2\sum_{i=1}^N\|\xi_i-\eta_i\|_\infty^2, \ \ \xi,\eta\in\C^N.
\end{split}\end{equation}
So, \eqref{ccb}, \eqref{ccB} and {\bf(H$\sigma$)} yield that {\bf(A1)}-{\bf(A3)} hold for $\hat{b}, \hat{B}, \hat{\sigma},N$ replacing $F,H,G,d$ respectively. Therefore, according to Theorem \ref{Existence}, for each $N\ge1$, \eqref{B1} and consequently \eqref{eq4} admits a unique strong solution $\{(X^{i,N}(t))_{t\in[-r,T]}\}_{1\leq i\leq N}$.
%satisfying
%\begin{equation}\label{etsgs}\E \sup_{t\in[-r,T]}|X^N(t)|\le C(T)(1+N\E\|X_0\|_{\infty}).\end{equation}
Finally, by It\^{o}'s formula, \eqref{mon}, {\bf(HB)} and {\bf (H$\sigma$)}, there exists a constant $C>0$ such that
\begin{equation*}
\begin{split}
|X^{i,N}(t)|^2&\leq |X^{i}(0)|^2+C\int_0^t\left[1+|X^{i,N}(s)|^2+\ff{1}{N}\sum_{j=1}^N|X^{j,N}(s)|^2\right]\d s\\
&+C\int_0^t\|X^{i,N}_s\|_{L^1(m)}^2\d s+\int_0^t 2X^{i,N}(s)\si(s,X^{i,N}(s)) \d W^{i}(s).
\end{split}
\end{equation*}
Using the same argument to derive \eqref{SEM}, we arrive at
\begin{equation*}
\begin{split}
\E(\sup_{t\in[-r,T]}|X^{i,N}(t)|^2|\F_0)\leq C(T)(1+\|X_0^i\|_\infty^2).
\end{split}
\end{equation*}
This implies \eqref{ets} by Jensen's inequality with respect to conditional expectation.
\end{proof}

Finally, we give the quantitative propagation of chaos.

\begin{thm}\label{POC} Assume that $\E\|X_0^i\|_\infty^{p}<\infty$ for some $p>1$ and $p\neq 2$. Let $\mu_t=\L_{X^i(t)}$.
\begin{enumerate}
\item[(1)]
Then there exists a constant $C(p,T)>0$ depending only on $p,T$ such that
\begin{equation}\label{EXY}
\sup_{t\in[0,T]}\E\big|X^i(t)-X^{i,N}(t)\big|\le C(p,T)(1+\left(\E\|X_0^i\|_\infty^p\right)^{\frac{1}{p}})(N^{-1/2}+N^{-\frac{p-1}{p}}),
\end{equation}
%where
%\begin{equation*}\begin{split}
%R_n(N)=
%\begin{cases}
%N^{-\ff{1}{2}}+N^{-\frac{q-\theta}{q}},~~~~~~~~~~~~~~~~~~~~\theta>\frac{1}{2}, q\neq 2\theta,\\
%N^{-\ff{1}{2}}\log (1+N)+N^{-\frac{q-\theta}{q}},~~~~ ~~\theta=\frac{1}{2}, q\neq 2\theta,\\
%N^{-\ff{\theta}{n}}+N^{-\frac{q-\theta}{q}},~~~~~~~~~~~~~~~~~~~\theta\in[1,\frac{1}{2}), q\neq \frac{n}{n-\theta},
%\end{cases}
%\end{split}\end{equation*}
and consequently,
\begin{equation}\begin{split}\label{WAR}
&\sup_{t\in[0,T]}\E\W_1(\hat{\mu}_t^N,\mu_t)\le C(p,T)(1+\left(\E\|X_0^i\|_\infty^p\right)^{\frac{1}{p}})(N^{-1/2}+N^{-\frac{p-1}{p}}).
\end{split}\end{equation}
\item[(2)] If in addition, $\sigma^2\geq \delta$ for some $\delta>0$ and there exists a constant $K\geq 0$ such that
\begin{equation}\begin{split}\label{sig} &|b(t,x,\mu)-b(t,x,\nu)|+|B(t,\xi,\mu)-B(t,\xi,\nu)|\\
&\qquad\quad\leq K (1\wedge\W_1(\mu,\nu)),\ \ \mu,\nu\in\scr P_1(\R), t\in[0,T], x\in\R,\xi\in\C.
\end{split}\end{equation}
then there exists a constant $C(p,T)>0$ depending only on $p,T$ such that for any $1\leq k\leq N$,
\begin{equation*}\begin{split}
&\sup_{t\in[0,T]}\|\L_{(X^{1,N}(t),X^{2,N}(t),\cdots,X^{k,N}(t))}-\mu_t^{\otimes k}\|^2_{var}\\
&\leq 2\sup_{t\in[0,T]}\Ent\left(\mu_t^{\otimes k}|\L_{(X^{1,N}(t),X^{2,N}(t),\cdots,X^{k,N}(t))}\right)\\
%&\le Cl R_n(N)1_{\{\theta\in[1,2)\}}+ClR_n(N)^{\frac{2}{\theta}}1_{\{\theta\geq 2\}}.
&\leq kC(p,T)(1+\left(\E\|X_0^i\|_\infty^p\right)^{\frac{1}{p}})(N^{-\frac{1}{2}}+N^{-\frac{(p-1)}{p}}),
\end{split}\end{equation*}
where $\mu_t^{\otimes k}=\prod_{i=1}^k\mu_t$, the $k$-independent product of $\mu_t$.
\end{enumerate}
\end{thm}
\begin{proof}
Applying It\^o's formula, it holds
\begin{equation*}
\begin{split}
&V_\vv(X^i(t)-X^{i,N}(t))\\
&=\int_0^tV_\vv'(X^i(s)-X^{i,N}(s))\big(b(s,X^i(s),\hat{\mu}^N_s)-b(s,X^{i,N}(s),\mu_s)\big)\d s\\
&\quad+\int_0^tV_\vv'(X^i(s)-X^{i,N}(s))\big(B(s,X^i_s,\hat{\mu}^N_s)-B(s,X^{i,N}_s,\mu_s)\big)\d s\\
&\quad+\ff{1}{2}\int_0^tV_\vv''(X^i(s)-X^{i,N}(s))\big(\si(s,X^i(s))-\si(s,X^{i,N}(s))\big)^2\d s\\
&\quad+\int_0^tV_\vv'(X^i(s)-X^{i,N}(s))\big(\si(s,X^i(s))-\si(s,X^{i,N}(s))\big)\d W^i(s).
\end{split}
\end{equation*}
By the same argument to derive \eqref{W11} and adopting the triangle inequality for $\mathbb{W}_1$, we arrive at
\begin{equation*}
\E|X^i(t)-X^{i,N}(t)|\le
C\int_0^t\big\{\E|X^i(s)-X^{i,N}(s)|+\E\mathbb{W}_1(\mu_s,\tt\mu_s^N)+\E\mathbb{W}_1(\tt\mu_s^N,\hat\mu_s^N)\big\}\d
s,
\end{equation*}
where $\tt\mu^N$ was introduced in \eqref{H1}. By \cite[Theorem 1]{FG}, there exists a constant $C(p,T)>0$ such that
\begin{equation}\label{w2}
\E\mathbb{W}_1(\mu_t,\tt\mu_t^N)\le C(p,T)(1+(\E\|X^i_0\|^p_\infty)^{\frac{1}{p}}) (N^{-1/2}+N^{-\frac{p-1}{p}}).
\end{equation}
As a result, it follows from \eqref{W3} and \eqref{w2}
that
\begin{equation*}
\begin{split}
&\E|X^i(t)-X^{i,N}(t)|\\
&\le
C_1\int_0^t\Big\{\E|X^i(s)-X^{i,N}(s)|+C(p,T)(1+(\E\|X^i_0\|^p_\infty)^{\frac{1}{p}}) (N^{-1/2}+N^{-\frac{p-1}{p}})\Big\}\d
s
\end{split}
\end{equation*}
for some constant $C_1>0$.
Consequently, we derive \eqref{EXY} by \eqref{FMT}, \eqref{ets} and Gr\"{o}nwall's inequality.
Finally, note that
\begin{equation*}
\begin{split}
\W_1(\hat\mu^N _s,\mu_s)
&\le
\mathbb{W}_1(\hat\mu^N_s,\tt\mu^N_s) +\mathbb{W}_1(\tt\mu^N_s,\mu_s)\le
\frac{1}{N}\sum_{i=1}^N|X^{i,N}(s)-X^i(s)| +\mathbb{W}_1(\tt\mu^N_s,\mu_s),
\end{split}
\end{equation*}
which together with \eqref{EXY} and \eqref{w2} yields \eqref{WAR}.

(2) Rewrite \eqref{eq4} as
$$\d X^{i,N}(t)=b(t,X^{i,N}(t), \mu_t)\d t+B(t,X^{i,N}_t, \mu_t)\d t+\sigma(t,X^{i,N}(t))\d \tilde{W}^i(t), \ \ 1\leq i\leq k$$
with
$$\d \tilde{W}^i(t)=\tilde{\Gamma}^i(t)\d t+\d W^i(t),\ \ 1\leq i\leq k$$
and $$\tilde{\Gamma}^i(t)=\sigma(t,X^{i,N}(t))^{-1}[b(t,X^{i,N}(t), \hat\mu_t^N)-b(t,X^{i,N}(t), \mu_t)+B(t,X^{i,N}_t, \hat\mu_t^N)-B(t,X^{i,N}_t, \mu_t)].$$
It follows from \eqref{sig} and $\sigma^2\geq \delta$ that there exists a constant $C>0$ such that
\begin{align}\label{gam}|\tilde{\Gamma}^i(t)|\leq C(\W_{1}(\hat\mu_t^N,\mu_t)\wedge 1),\ \ t\in[0,T],1\leq i\leq k.\end{align}
Let $$R^k_t=\exp{\left\{-\sum_{i=1}^k\int_0^t\<\tilde{\Gamma}^i(s),\d W^i(s)\>-\frac{1}{2}\sum_{i=1}^k\int_0^t|\tilde{\Gamma}^i(s)|^2\d s \right\}},\ \ t\in[0,T].$$
\eqref{gam} and Girsanov's theorem imply that $\{R^k_t\}_{t\in[0,T]}$ is a martingale and $((\tilde{W}^i(t))_{1\leq i\leq k})_{t\in[0,T]}$ is a $k$-dimensional Brownian motion under $\Q_T^k=R^k_T \P$ and
\begin{align}\label{QQS}\L_{(X^{1,N}(t),X^{2,N}(t),\cdots,X^{k,N}(t))}|\Q^k_T=\mu_t^{\otimes k},\ \ t\in[0,T].
\end{align}
This implies that
\begin{align*}\mu_t^{\otimes k}(f)&=\E [R^k_T f(X^{1,N}(t),X^{2,N}(t),\cdots,X^{k,N}(t))]\\
&=\E [R^k_t f(X^{1,N}(t),X^{2,N}(t),\cdots,X^{k,N}(t))],\ \ f\in\scr B_b(\R^{ k}), t\in[0,T].
\end{align*}
So,  there exists  a constant $C>0$ such that
\begin{align*}&\Ent(\mu_t^{\otimes k}|\L_{(X^{1,N}(t),X^{2,N}(t),\cdots,X^{k,N}(t))})\\
&= \E(R^k_t\log R^k_t)= \frac{1}{2}\sum_{i=1}^k\int_0^t\E^{\Q_T^k}|\tilde{\Gamma}^i(s)|^2\d s\leq  C^2k\int_0^t\E^{\Q_T^k}(\W_{1}(\hat\mu_s^N,\mu_s)\wedge 1)^2\d s,\ \ t\in[0,T].\end{align*}
This together with Pinsker's inequality \eqref{ETX} yields
\begin{align*}\nonumber&\|\mu_t^{\otimes k}-\L_{(X^{1,N}(t),X^{2,N}(t),\cdots,X^{k,N}(t))}\|_{var}^2\\
&\leq 2\Ent(\mu_t^{\otimes k}|\L_{(X^{1,N}(t),X^{2,N}(t),\cdots,X^{k,N}(t))})\\
\nonumber&\leq 2C^2k\int_0^t\E^{\Q_T^k}(\W_{1}(\hat\mu_s^N,\mu_s))\d s.
\end{align*}
The proof is finished by \eqref{QQS} and \eqref{w2}.
\end{proof}
\begin{rem}\label{POC} For quantitative propagation of chaos, one can refer to \cite{HRZ} and references therein for the convolution type distribution dependent SDEs. Since we only assume that the drift is Lipschitz continuous under $L^1$-Wasserstein distance and the estimate in \cite[Theorem 1]{FG} for the convergence rate of empirical distribution of i.i.d. random variables plays crucial role, the order of the quantitative propagation of chaos may be not optimal.
\end{rem}
%\smallskip
%
%\noindent{\bf Acknowledgment }
%
%\smallskip
%
%
%We are indebted to the referee for his/her valuable comments which
%have greatly improved our earlier version.

\beg{thebibliography}{99}
%\bibitem{AZ0} G. B. Arous, O. Zeitouni, Increasing propagation of chaos for mean field models, \emph{Ann. Inst. Henri Poincar\'e Probab. Stat.} 35(1999), 85-102.

%\bibitem{BHY} Bao, J., Huang, X., Yuan, C., Convergence rate of Euler-Maruyama scheme for SDEs
%with H\"older-dini Continuous drifts, {\it J.  Theor. Probab.},
%{\bf 32} (2019),  848--871.

\bibitem{BMP} M. Bauer, T. M-Brandis, F. Proske, \emph{Strong Solutions of Mean-Field Stochastic Differential Equations with irregular drift,} Electron. J. Probab. 23(2018), 1-35.

%\bibitem{BLM}Buckdahn, R., Li, J., Ma, J., A mean-field stochastic
%control problem with partial observations, {\it  Ann. Appl.
%Probab.}, {\bf27} (2017),  3201--3245.

% \bibitem{BLPR}Buckdahn, R., Li, J., Peng, S., Rainer, C.,
 % Mean-field stochastic differential equations and associated PDEs, {\it Ann. Probab.}, {\bf 45} (2017),  824--878.

\bibitem{BO} R. J. Berman, M. \"{O}nnheim, \emph{Propagation of Chaos for a Class of First Order Models with Singular Mean Field Interactions,} SIAM J. Math. Anal. 51(2019), 159-196.

\bibitem{CD} R. Carmona, F. Delarue,  \emph{Probabilistic theory of mean
field games with applications. I.} Mean field FBSDEs, control, and
games. Probability Theory and Stochastic Modelling, 83. Springer,
Cham, 2018.
\bibitem{Ch} P.-E. Chaudru de Raynal, \emph{Strong well-posedness of McKean-Vlasov stochastic differential equation with H\"older drift, } Stochastic Process. Appl.  130(2020),     79-107.

\bibitem{CF} P.-E. Chaudru de Raynal, N. Frikha, \emph{Well-posedness for some non-linear diffusion processes and related pde on the Wasserstein space,} J. Math. Pures Appl. 159(2022), 1-167.

%\bibitem{CST}   J.-F. Chassagneux,  L. Szpruch, A. Tse, \emph{Weak quantitative propagation of chaos via differential calculus on the space of
%measures,} arXiv:1901.02556.

%\bibitem{CW}Christoph, R., Wolfgang, S., An adaptive Euler-Maruyama scheme for McKean SDEs with super-linear growth and application to the mean-field FitzHugh-Nagumo model, arXiv:2005.06034.

%\bibitem{CIR}Cox, J.~C., Ingersoll, J.~E., Ross, S.~A.,  An intertemporal general equilibrium model of asset
%prices, {\it  Econometrica}, {\bf 53} (1985),  363--384.

%\bibitem{CM} Crisan, D., McMurray, E., Smoothing properties of
%McKean-Vlasov SDEs, {\it Probab. Theory Related Fields}, {\bf 171}
%(2018),   97--148.

\bibitem{DEG} G. dos Reis, S. Engelhardt, G. Smith, \emph{Simulation of McKean-Vlasov SDEs with super linear
growth,} IMA J. Numer. Anal. 42(2022), 874-922.

\bibitem{FG} N. Fournier, A. Guillin, \emph{On the rate of convergence in Wasserstein distance of the empirical measure,} Probab. Theory Related Fields 162(2015), 707-738.
\bibitem{FH} N. Fournier, M. Hauray, \emph{Propagation of Chaos for the Landau equation with moderately soft potentials,} Ann. Probab. 44(2016), 3581-3660.
\bibitem{GLL} O. Gu\'{e}ant O, J.-M. Lasry, P.-L. Lions, Mean Field Games and Applications, \emph{Paris-Princeton Lectures on Mathematical Finance 2010. Lecture Notes in Math.} 2003, Springer, Berlin, 205-266.
\bibitem{GLWZ} A. Guillin, W. Liu,  L. Wu, C. Zhang, \emph{The kinetic Fokker-Planck equation with mean field interaction,} J. Math. Pures Appl. 150(2021),1-23.
%\bibitem{EGZ}Eberle, A., Guillin, A., Zimmer, R.,
%Quantitative Harris-type theorems for diffusions and McKean-Vlasov
%processes, {\it Trans. Amer. Math. Soc.}, {\bf 371} (2019),
%7135--7173.

% \bibitem{GR}Gy\"ongy, I., R\'{a}sonyi, M., A note on Euler approximations for SDEs with H\"older continuous diffusion
% coefficients,
%  {\it Stochastic Process. Appl.}, {\bf121} (2011),   2189--2200.
\bibitem{HRZ} Z. Hao, M. R\"{o}ckner, X. Zhang, \emph{Strong convergence of propagation of chaos for McKean-Vlasov SDEs with singular interactions,} arXiv:2204.07952.
\bibitem{HW} X. Huang, F.-Y.  Wang,  \emph{Distribution dependent SDEs with singular coefficients, } Stochastic Process. Appl. 129(2019), 4747-4770.

\bibitem{HW20} X. Huang, F.-Y.  Wang,  \emph{McKean-Vlasov SDEs with drifts discontinuous under Wasserstein distance,} Discrete Contin. Dyn. Syst. 41(2021), 1667-1679.

\bibitem{HX} X. Huang,  \emph{Path-distribution dependent SDEs with singular coefficients,}  Electron. J. Probab. 26(2021), 1-21.
\bibitem{JR} B. Jourdain, J. Reygner, \emph{Propagation of chaos for rank-based interacting diffusions and long time behaviour of a scalar quasilinear parabolic equation,}  Stoch. PDE: Anal. Comp. 1(2013), 455-506.

\bibitem{IW} N. Ikeda, S. Watanabe, \emph{Stochastic differential equations and diffusion processes}, North-Holland Mathematical Library, 24.
 North-Holland Publishing Co., Amsterdam-New York; Kodansha, Ltd., Tokyo, 1981.

\bibitem{L} D. Lacker, \emph{On a strong form of propagation of chaos for McKean-Vlasov equations,} Electron. Commun. Probab. 23(2018), 1-11.

%\bibitem{MS} Mehri, S.,      Stannat, W.,  Weak solutions to Vlasov-McKean equations under
%Lyapunov-type conditions,   {\it Stoch. Dyn.}, {\bf 19} (2019),   1950042.

\bibitem{Mc} H. P. McKean, \emph{A class of Markov processes associated with nonlinear parabolic equations,} Proc Natl Acad Sci U S A, 56(1966), 1907-1911.
%\bibitem{MB}Mezerdi, M.~A., Bahlali, K.,   Khelfallah, N.,  Mezerdi, B.,
%Approximation and generic properties of McKean-Vlasov stochastic equations with continuous coefficients, arXiv: 1909.13699.

\bibitem{MV} Yu. S. Mishura, A. Yu. Veretennikov, \emph{Existence and uniqueness theorems for solutions of McKean-Vlasov stochastic equations,} Theor. Probability and Math. Statist. 103(2020), 59-101.

\bibitem{Pin} M. S. Pinsker, \emph{Information and Information Stability of Random Variables and Processes,} Holden-Day, San Francisco, 1964.

\bibitem{RS} M. von Renesse, M. Scheutzow, \emph{Existence and uniqueness of solutions of stochastic functional differential equations,} Random Oper. Stoch. Equ. 18(2010), 267-284.

\bibitem{RW}  P. Ren, F.-Y. Wang, \emph{Exponential convergence in entropy and Wasserstein for McKean-Vlasov SDEs,}  Nonlinear Anal. 206(2021), 112259.

\bibitem{RZ} M. R\"{o}ckner, X. Zhang, \emph{Well-posedness of distribution dependent SDEs with singular drifts,} Bernoulli 27(2021), 1131-1158.

\bibitem{SZ} A.-S. Sznitman, Topics in propagation of chaos, \'{E}cole d' \'{E}t\'{e} de Prob. de Saint-Flour XIX-1989, 165-251, Lecture Notes in Math., 1464, Springer, Berlin, 1991.

\bibitem{Wangb} F.-Y. Wang, \emph{Distribution-dependent SDEs for Landau type equations,} Stochastic Process. Appl. 128(2018), 595-621.

\bibitem{W21} F.-Y. Wang, \emph{Exponential Ergodicity for Singular Reflecting McKean-Vlasov SDEs,} arXiv:2108.03908.

\bibitem{Zhang}  X. Zhang, \emph{A discretized version of Krylov's estimate and its
applications,} Electron. J.
Probab. (24)2019, 1--17.

\end{thebibliography}

\end{document}